\documentclass{amsart}

\newif\ifdraft\draftfalse
\newif\ifcite\citefalse
\newif\ifblow\blowtrue

\usepackage{amsfonts,amssymb, amsmath}
\usepackage{oldgerm}
\usepackage{url}
\usepackage{amscd}
\usepackage{mathrsfs} 
\usepackage[hyperindex]{hyperref}  
\usepackage[alphabetic]{amsrefs}
\ifdraft\ifcite\usepackage{showkeys}\else\usepackage[notcite,notref]{showkeys}\fi\fi

\newtheorem{proposition}[equation]{Proposition}
\newtheorem{theorem}[equation]{Theorem}

\newtheorem{lemma}[equation]{Lemma}

\theoremstyle{definition}
\newtheorem{definition}[equation]{Definition}

\theoremstyle{remark}

\newtheorem{remark}[equation]{Remark}

\newtheorem{rem}[equation]{Remark}
\newtheorem{question}[equation]{Question}

\numberwithin{equation}{section}

\def\bc{\begin{cases}}
\def\ec{\end{cases}}

\def\ol{\overline}
\def\ul{\underline}

\def\a{\alpha}

\def\g{\gamma}
\def\th{\theta}

\def\t{\tilde}

\def\ch{{\mathcal H}}

\def\cl{{\mathcal L}}

\def\bc{{\mathbb C}}

\def\bp{{\mathbb P}}

\def\br{{\mathbb R}}

\def\bz{{\mathbb Z}}

\def\er{\eqref}

\def\bz{\mathbb Z}

\def\br{\mathbb R}
\def\bc{\mathbb C}
\def\bp{\mathbb P}

\def\lp2{L_pH_{2p}}
\def\bean{\begin{eqnarray}}
\def\eean{\end{eqnarray}}
\def\bea{\begin{eqnarray*}}
\def\eea{\end{eqnarray*}}
\def\beq{\begin{equation}}
\def\eeq{\end{equation}}
\def\beq*{\begin{equation*}}
\def\eeq*{\end{equation*}}
\def\bal{\begin{align*}}
\def\eal{\end{align*}}
\def\baln{\begin{align}}
\def\ealn{\end{align}}

\def\beg{\begin{gather*}}
\def\eng{\end{gather*}}
\def\bqu{\begin{question}}
\def\equ{\end{question}}

\def\nm{\nonumber}

\def\ban{\begin{proof}[Answer]}
\def\ean{\end{proof}}

\def\ra{\Rightarrow}

\def\p{\partial}

\def\on{\operatorname}

\def\bqu{\begin{question}}
\def\equ{\end{question}}

\def\0110{\begin{matrix} 0 & 1\\1&0\end{matrix}}

\def\t{\tilde}

\def\ban{\begin{proof}[Answer]}
\def\ean{\end{proof}}

\def\ben{\begin{equation}}
\def\een{\end{equation}}

\def\la{\langle}
\def\ra{\rangle}

\def\j1{{(j+1)}}

\def\f32{{}_3F_2}

\newcommand\lf{\lfloor}
\newcommand\rf{\rfloor}
\newcommand{\bnu}{\boldsymbol \nu}
\newcommand{\bsg}{\boldsymbol \sigma}
\newcommand{\bzt}{\boldsymbol \zeta}

\newcommand{\ginf}{\gamma_\infty}

\begin{document}

\title
{Toda systems and hypergeometric equations}

\author{Chang-Shou Lin}
\email{cslin@math.ntu.edu.tw}
\address{Department of Mathematics, Taida Institute of Mathematical Sciences, National Taiwan University, Taipei 106, Taiwan}

\author{Zhaohu Nie}
\email{zhaohu.nie@usu.edu}
\address{Department of Mathematics and Statistics, Utah State University, Logan, UT 84322-3900, USA}

\author{Juncheng Wei}
\email{jcwei@math.ubc.ca}
\address{Department of Mathematics, University of British Columbia, Vancouver, BC V6T 1Z2, Canada}


\begin{abstract} This paper establishes certain existence and classification results for solutions to $SU(n)$ Toda systems with three singular sources at 0, 1, and $\infty$. First, we determine the necessary conditions for such an $SU(n)$ Toda system to be related to an $n$th order hypergeometric equation. Then, we construct solutions for $SU(n)$ Toda systems that satisfy the necessary conditions and also the interlacing conditions from Beukers and Heckman \cite{BH}. Finally, for $SU(3)$ Toda systems satisfying the necessary conditions, we classify, under a natural reality assumption, that all the solutions are related to hypergeometric equations. This proof uses the Pohozaev identity.
\end{abstract}

\maketitle

\section{Introduction}

To a simple Lie group, there is associated a Toda system on the plane. 
The $SU(2)$ Toda system is just the classical Liouville equation \cite{Liouville}. 
The Toda systems and the Liouville equation arise in many physical and geometric problems. For example, in the Chern-Simons theory, the Liouville equation is related to the abelian gauge field theory, while the Toda systems are related to nonabelian gauges (see \cites{Yang, Taran}). On the geometric side, the Liouville equation is related to conformal metrics on the Riemann sphere $S^2$ with constant Gaussian curvature. The Toda systems are related to holomorphic curves in 
projective spaces \cite{Doliwa} and the Pl\"ucker formulas \cite{GH}, and the periodic Toda systems are related to harmonic maps \cite{Guest}. 

Singular sources can be introduced to the Toda systems, which correspond to conical singularities in conformal geometry and to vortices in gauge theory. 
The question of conformal metrics with conical singularities has been widely studied using various viewpoints. The works \cites{Troy, LT, CL} studied the solutions using the variational method by the Moser-Trudinger inequality.
In the case of three singularities on $S^2$, which can always be taken as 0, 1, and $\infty$ by a M\"obius transform,  
the works \cites{Erem, FUY} studied the existence and uniqueness of the conformal metrics with constant Gaussian curvature by analyzing the monodromy of the corresponding 2nd order hypergeometric equation. In this paper, we aim to generalize such a relationship to certain $SU(n)$ Toda systems with three singularities and $n$th order hypergeometric equations. 

We identify the plane $\br^2$ with the complex plane $\bc$ by $(x_1, x_2)\mapsto z = x_1 + i x_2$. Let $\bp^1 = \bc\cup \{\infty\}$ be the complex projective line. 
The $SU(n)$ Toda system with singular sources at $\{0, 1, \infty\}$ is the following system of semilinear elliptic PDEs on the plane in $(n-1)$ unknowns
\begin{equation}
\label{toda}
\begin{cases}
\displaystyle
\Delta u_i + 4\sum_{j=1}^{n-1} a_{ij} e^{u_j} = 4\pi (\gamma_{0,i} \delta_0 + \gamma_{1, i} \delta_1), & \gamma_{0, i}, \gamma_{1, i}>-1,\\
u_i (z) = -2\gamma_{\infty, i} \log|z| + O(1)\quad \text{near }\infty, & \gamma_{\infty, i} > 1,
\end{cases}
\end{equation}
where
$\delta_0$ and $\delta_1$ are the Dirac measure at 0 and 1, 
and the matrix $A = (a_{ij})$ is the Cartan matrix of $SU(n)$ of dimension $(n-1)$: 
\begin{equation}
\label{Cartan-A}
A =  \begin{pmatrix}
2 & -1 & & &\\
-1 & 2 & -1 & & \\
 & \ddots & \ddots & \ddots& \\
 & &-1 & 2 & -1\\
 & & & -1 & 2
 \end{pmatrix}.
 \end{equation}


One of our interests is to determine the range of the $\g_{\infty, i}$ in terms of the $\g_{0, i}$ and $\g_{1, i}$. 

When $n=2$, system \er{toda} becomes 
\begin{equation}\label{liouv}
\begin{cases}
\Delta u + 8 e^{u} = 4\pi (\gamma_0 \delta_0 + \gamma_1 \delta_1), & \gamma_0, \gamma_1 > -1,\\
u(z) = -2\gamma_\infty \log|z| + O(1)\quad \text{near }\infty, & \gamma_{\infty} > 1.
\end{cases}
\end{equation}
Geometrically the metric $4 e^u |dz|^2$ 
has constant Gaussian curvature 1 and conical singularities at 0, 1, and $\infty$  with positive angles 
$2\pi (\g_0 + 1)$, $2\pi (\g_1 + 1)$, and $2\pi (\g_\infty-1)$. 
The basic idea is that \er{liouv} is related to a second order Fuchsian ODE with regular singularities only at $0, 1$ and $\infty$. It is known that such ODEs can be completely classified by the Gauss hypergeometric equations. 
In \cite{Erem} and \cite{FUY}, they apply the monodromy theory of Gauss hypergeometric equations to completely solve the problem of the range of $\g_\infty$. 

System \er{toda} has the following version which is easier to study for many purposes. 
Set 
\begin{equation}
U_i = \sum_{j=1}^{n-1} a^{ij} u_j, \quad (a^{ij}) = A^{-1}.
\end{equation}
Then system \er{toda} is equivalent to  
\begin{equation}
\label{toda2}
\begin{cases}
\displaystyle
\Delta U_{i} + 4 e^{u_i} = 4\pi (\gamma^i_0 \delta_0 + \gamma^i_1\delta_1),\\
U_i (z) = -2\gamma^i_\infty \log|z| + O(1)\quad \text{near }\infty,
\end{cases}
\end{equation}
where 
\begin{equation}\label{def alpha}
\gamma^i_P = \sum_{j=1}^{n-1} a^{ij} \gamma_{P, j},\quad P\in V=\{0, 1, \infty\}.
\end{equation}
For later use, we also define 
\begin{align}
\label{mu}
\mu_{P, i} &= \gamma_{P, i} + 1 > 0, \quad P=0\text{ or }1,\\
\label{mu8}
\mu_{\infty, i} &= \gamma_{\infty, i} - 1 > 0. 
\end{align}

To any solution $U_i$ of \er{toda2}, there are associated the so-called $W$-invariants $W_j$, $2\leq j\leq n$, such that $f = e^{-U_1}$ satisfies 
an $n$th order complex ODE
\begin{equation}\label{Lin}
y^{(n)} + \sum_{j=2}^n W_j y^{(n-j)} = 0.
\end{equation}
This fact was proved in \cite{LWY} and \cite{BC}, and we will recall it in \er{ode}. The $W$-invariants $W_j$ are polynomials in the $\frac{\p^k U_i}{\p z^k}$ for $k\geq 1$ and are meromorphic functions in $\bc\cup \{\infty\}$ with 
poles only at $0, 1$ and $\infty$ (see Proposition \ref{prop-wj}). 

When $n\geq 3$, it is known that Fuchsian ODEs with three singular points can not be classified by just the local exponents at the singular points. But there is a class of the so-called hypergeometric equations which can be uniquely determined by its local exponents, that is, the monodromy transformations at 0, 1, and $\infty$ satisfy the \emph{local rigidity property}. 

With $\theta = z\frac{\p}{\p z}$, the $n$th order hypergeometric equation associated to parameters $\alpha=(\a_1, \dots, \a_n)$ and $\beta=(\beta_1, \dots,\beta_n)$ is 
\begin{equation}\label{dab}
D(\a; \beta)y = ((\theta + \beta_1 -1)\cdots (\theta + \beta_n -1) - z(\theta + \a_1) \cdots (\theta + \a_n)) y = 0.
\end{equation}
The local exponents of the hypergeometric equation are easily computed as  
\begin{align}
1-\beta_1, &\dots, 1-\beta_n &\text{at }z=0,\label{ghg 0}\\
\a_1, &\dots, \a_n,  &\text{at }z=\infty,\label{ghg 8}\\
0, 1, &\dots, n-2, \g=\sum_{j=1}^n \beta_j - \sum_{i=1}^n \a_i -1. &\text{at }z=1,\label{the c} 
\end{align}
We refer to \cite{BH} for the basic theory of hypergeometric equations. 

We remark that the ODE \er{Lin} associated to the Toda system needs a shift to become hypergeometric, because we at least need to ensure the appearance of 0 as a local exponent at $z=1$ in view of \er{the c}. Let $y(z)$ be a solution of \er{Lin}, and with a constant $\tau$, let 
$$
\t y(z) = (z-1)^\tau y(z).
$$
Then $\t y(z)$ also satisfies an $n$th order ODE.



Our first result gives the necessary conditions for the solutions of the $SU(n)$ Toda system \er{toda2} to be related to an $n$th order hypergeometric equation. 
\begin{theorem}\label{necessary} Assume that the $U_i$ are a set of solutions of \er{toda2}. 
If for some constant $\tau$, 
\begin{equation}\label{def gz}
g(z) = (z-1)^\tau e^{-U_1}
\end{equation}
satisfies 
a hypergeometric equation,
then the $\gamma_{1, i}$ 
and the corresponding $\tau$ belong to one of the following $n$ cases
\begin{align}
0)\ &\gamma_{1, 1}\neq 0,\quad \text{other }\gamma_{1, j}=0, && \tau= \gamma_1^1 -\mu_{1, 1}; \label{1neq0}\\
i)\ &\g_{1, i} + \g_{1, i+1} = -1,\ 1\leq i\leq n-2,\quad \text{other }\gamma_{1, j}=0, &&\tau=\g_1^1;\label{l&f}\\
n-1)\ &\g_{1, n-1}\neq 0,\quad\text{other }\gamma_{1, j}=0. &&\tau = \g_1^1.\label{cond0} 
\end{align}
\end{theorem}

Let us now recall the concept of monodromy for a complex ODE in the unknown $y=y(z)$ on the complex plane $\bc$
\begin{equation}\label{gen ode}
y^{(n)} + a_1(z) y^{(n-1)} + \cdots + a_n(z) y = 0,
\end{equation}
whose coefficients $a_i(z)$ are meromorphic functions with poles at 0 and 1. 
Fix a base point $z_0\in \bc\setminus \{0, 1\}$, e.g. $z_0 = \frac{1}{2}$, and let $V$ be the $n$-dimensional vector space of local solutions of \er{gen ode} around $z_0$.  
Let $G = \pi_1(\bp^1\setminus V, z_0) = \pi_1(\bc\setminus \{0, 1\}, z_0)$ be the fundamental group,  and denote by 
\begin{equation}\label{monod}
\phi: G\to GL(V)
\end{equation}
the monodromy representation which is defined through analytic continuation of solutions. 
The most useful result in \cite{BH} to us is the following.

\begin{definition}\cite{B} We say that two sets of $n$ real numbers $\{\a_1, \dots, \a_n\}$ and $\{\beta_1, \dots, \beta_n\}$ \emph{interlace modulo $\bz$} if the set $\{\a_1 - \lf \a_1\rf, \dots, 
\a_n - \lf \a_n \rf\}$ and $\{\beta_1 - \lf \beta_1\rf, \dots, \beta_n - \lf \beta_n \rf\}$ interlace on $[0, 1)$, that is, the two sets are disjoint and their elements occur alternately in increasing order. 
\end{definition}

\begin{theorem}\cite{BH}\label{pos-inv-Herm} If the $\alpha_j$ and the $\beta_k$ are real and $\a_j-\beta_k\notin \bz$ for all $j, k=1,\dots, n$, then there is a positive-definite Hermitian form on $V$ invariant under the monodromy representation \er{monod} of the hypergeometric equation \er{dab} if and only if the $\a_j$ and the $\beta_k$ interlace mod $\bz$. 
\end{theorem}

Using this theorem, we have the following existence result. 
\begin{theorem}\label{hpg exist} Assume that the Toda system \er{toda} satisfies one of the necessary conditions in Theorem \ref{necessary} with $\tau$ determined correspondingly. 
Define
\begin{equation}
\label{abs}
\begin{split}
(\a_1,\a_2, \dots, \a_n) &= (-\tau-\gamma^{1}_{\infty}, -\tau-\gamma^{1}_{ \infty} + \mu_{\infty, 1},  \cdots, -\tau -\gamma^{1}_{ \infty} + \mu_{\infty, 1} + \cdots + \mu_{\infty, n-1}),\\
(\beta_1, \cdots, \beta_{n-1}, \beta_n) &= (1 +\g_0^1-\mu_{0, 1}  - \cdots - \mu_{0, n-1}, \dots, 1+\g_0^1-\mu_{0, 1}, 1+\g_0^1).
\end{split}
\end{equation}
If the $\alpha_j$ and the $\beta_k$ interlace mod $\bz$, then the Toda system \er{toda} has a solution. 
\end{theorem}

We note that the interlacing conditions for $\alpha$ and $\beta$ actually give a set of inequalities on the $\g_{\infty, i}$ in terms of the $\g_{0,i}$ and $\g_{1, i}$. 

Next we study under the necessary conditions in Theorem \ref{necessary}, whether all solutions of the Toda system are 
related to a hypergeometric equation. 
We are able to prove that this is the case for $SU(3)$ Toda systems under a natural reality assumption. 

\begin{theorem}\label{thm-su3} For $SU(3)$ Toda systems and under the assumption that the coefficient $F$ in the Fuchsian ODE \er{eqsu3} is real, 
the necessary conditions in Theorem \ref{necessary} are also sufficient, that is, the corresponding $g(z)$ in \er{def gz} satisfies a hypergeometric equation. 
\end{theorem}

\medskip
\noindent{\bf Acknowledgments.} We thank Dr. Wen Yang for his help with the analysis in Lemma \ref{lem-radial} and Proposition \ref{finer}. Z. Nie thanks the University of British Columbia, Wuhan University, and the National Taiwan University for hospitality during his visits in 2015 and 2016, where part of this work was done. He also acknowledges the Simons Foundation through Grant \#430297. The research of J. Wei is partially supported by NSERC of Canada.

\section{Fuchsian equations and hypergeometric equations}
\label{sect-Fuchs}

Our method of studying the solutions to the Toda system \er{toda2} is through a related complex ODE 
whose coefficients are the so-called $W$-invariants. This method was initiated in \cite{LWY} with some simplification from \cite{BC}. In Proposition \ref{prop-wj}, we show that 
the $W$-invariants take simple forms by our assumptions on the singular sources. As a simple consequence, the complex ODE 
is a Fuchsian equation with regular singularities at $0, 1$ and $\infty$. In Theorem \ref{prop-exp}, we compute the exponents of the equation at the singular points in terms of the strength of the singularities.  
At the end of this section, we prove the necessary conditions in Theorem \ref{necessary} for the Fuchsian ODE of our Toda system to be become a hypergeometric equation after a shift. 

By $\Delta = 4\p_z\p_{\bar z} = 4\frac{\p}{\p z}\frac{\p}{\p \bar z}$ (this $4$ is the origin of the slightly unconventional 4 at the left hand side of \er{toda}), the first equation in \er{toda2} is the same as 
\begin{equation}\label{U in zz}
U_{i,z \bar z} + \exp\Big(\sum_{j=1}^{n-1}  a_{ij} U_j\Big)  = 0\quad \text{on }\bc\setminus \{0, 1\}.
\end{equation}
It is this form that is usually called the Toda field theory \cite{LS-book}. 

For a set $(U_1, \cdots, U_{n-1})$ of solutions to the Toda system \er{toda2}, consider the following linear ordinary differential operator
\begin{equation}
\label{my way}
\begin{split}
\cl &= (\p_z - U_{n-1, z}) (\p_z + U_{n-1, z} - U_{n-2, z}) \cdots (\p_z + U_{2, z} - U_{1, z}) (\p_z + U_{1, z}) \\
&=  \p_z^{n} + \sum_{j=2}^n W_j \p_z^{n-j}.
\end{split}
\end{equation}
Let $f=e^{-U_1}$. Clearly by the last factor $ (\p_z + U_{1, z})$, we have 
\begin{equation}\label{ode}
\cl f = \Big(\p_z^n  + \sum_{j=2}^n W_j \p_z^{n-j}\Big) f= 0.
\end{equation}

The $W_j$ are polynomials of the derivatives of the $U_i$ with respect to $z$. By \cite{N2}, they have the crucial property that 
\begin{equation}\label{char prop}
W_{j, \bar z}=0, \quad 2\leq j\leq n,
\end{equation}
if the $U_i$ are solutions of the Toda system \er{toda2}. 
The $W_j$ are called characteristic integrals of the Toda system in \cite{N2}, and they are also called $W$-symmetries or $W$-invariants in \cite{LWY} in view of their relationship to the $W$-algebras  \cite{Feher}. 

For a differential monomial in the $U_i$, we call by its \emph{degree} the sum of the orders of differentiation multiplied by the algebraic degrees of the corresponding factors. For example, the $W$-invariant for the Liouville equation $U_{z\bar z} + e^{2U} = 0$ is $U_{zz} - U_{z}^2$, and it has a homogeneous degree $2$. 
From Eq. \er{my way}, we see that $W_j$ has a homogeneous degree $j$ for $2\leq j\leq n$. 

\begin{proposition}\label{prop-wj} The $W$-invariants $W_j$ of the Toda system \er{toda2} have the following form
\begin{equation}\label{Wj now}
W_j = \sum_{k=0}^j \frac{A_{j, k}}{z^{j-k} (z-1)^k},\quad 2\leq j\leq n,
\end{equation}
where $A_{j,k}$ are  constants. 
\end{proposition}

\begin{proof} This proof is an adaption of the proof in \cite{LWY} of the corresponding assertion in their Eq. (5.9). 
First, by a standard integration argument applied to \er{toda2} we have
\begin{equation}\label{integ}
\int_{\br^2} e^{u_i} \,dx = {\pi}(\gamma^i_0 + \gamma^i_1 + \gamma^i_\infty),\quad 1\leq i\leq n-1.
\end{equation}

	Following \cite{LWY}*{Eq. (5.10)}, introduce 
\begin{equation}\label{def V}
V_i = U_i - 2\gamma^i_0 \log |z| - 2\g^i_1 \log|z-1|, \quad 1\leq i\leq n-1.
\end{equation}
Then system \er{toda2} becomes
\begin{equation}\label{eq for V}
\begin{cases}
\displaystyle\Delta V_i = -4 |z|^{2\gamma_{0,i}} |z-1|^{2\g_{1, i}} \exp\Big(\sum_{j=1}^{n-1} a_{ij} V_j\Big),\\
\displaystyle \int_{\br^2} |z|^{2\gamma_{0,i}} |z-1|^{2\g_{1, i}} \exp\Big(\sum_{j=1}^{n-1} a_{ij} V_j\Big) \, dx <\infty,
\end{cases}
\end{equation}
where the integrals are finite in view of \er{integ}. 
As $\gamma_{P,i} > -1$ for $P=0,1$, applying Brezis-Merle's argument in \cite{BM}, we have that 
\begin{equation}\label{important}
V_i \in C^{0, \alpha}  \quad\text{on }\bc\text{ for some }\alpha \in (0, 1)
\end{equation} 
and that they are upper bounded over $\bc$. 
Therefore we can express $V_i$ by the integral representation formula, and we have from \cite{LWY}*{Eq. (5.11)}
\begin{equation}
\label{order of V}
\begin{split}
\p_z^k V_i(z) &= O(1+|z|^{2(\gamma_{0, i} + 1)-k}) \quad\text{ near } 0,\\
\p_z^k V_i(z) &= O(1+|z-1|^{2(\gamma_{1, i} + 1)-k}) \quad\text{ near } 1,\\
\p_z^k V_i(z) &= O(|z|^{-k}) \quad\text{ near } \infty,\quad \forall\, k\geq 1.
\end{split}
\end{equation}


Therefore from \er{def V}, we have 
\begin{equation}
\label{order of U}
\begin{split}
\p_z^k U_i(z) &= O(|z|^{-k}) \quad\text{ near } 0, 
\\
\p_z^k U_i(z) &= O(|z-1|^{-k}) \quad\text{ near } 1, 
\\
\p_z^k U_i(z) &= O(|z|^{-k}) \quad\text{ near } \infty,\quad \forall\, k\geq 1.
\end{split}
\end{equation}
	By $W_{j, \bar z} = 0$ from \er{char prop} and that $W_j$ has degree $j$ for $j\geq 2$, we see from the above estimates that $z^{j}(z-1)^{j} W_j$ has removable singularities at 0 and 1 and hence is an entire function. 
	The order of $ z^{j}(z-1)^{j} W_j$ at $\infty$ is $O(|z|^j)$, and so it is a polynomial of degree $j$. 
	As such, we have
$$
z^j(z-1)^j W_j = \sum_{k=0}^j A_{j, k} z^k (z-1)^{j-k},
$$
where the coefficients $A_{j, k}$ can be successively computed by expanding $(z-1)^{j-k}$. 
This proves \er{Wj now}. 
\end{proof}


\begin{theorem}\label{prop-exp} The ODE \er{ode} is Fuchsian with regular singularities at 0, 1, and $\infty$. 
Its exponents are expressed in terms of the $\gamma_{P, i}$ as in \er{exponents at 0}, \er{exponents at 1} and \er{exponents at 8}. 
\end{theorem}

\begin{proof} By the forms of the $W_j$ in \er{Wj now}, the ODE \er{ode} has three regular singularities at 0, 1 and $\infty$ and is thus Fuchsian. Now we compute its exponents. 

By \er{def V}, we have 
\begin{equation}\label{uvz}
U_{i, z} = V_{i, z} + \frac{\g^i_0}{z} + \frac{\g^i_1}{z-1}.
\end{equation}
In the expansion \er{my way} to compute the ODE, the $W_j$ are polynomials of the $\p^k_z U_i$ for $k\geq 1$, and therefore they are polynomials involving $V_{i, z}$, $\frac{\g^i_0}{z}$, $\frac{\g^i_1}{z-1}$ and their derivatives. 

At point $z=0$, we claim that $\displaystyle \lim_{z\to 0} z^j W_j$ only depend on the terms of $W_j$ containing the $\frac{\g^i_0}{z}$ and their derivatives. 
Indeed, if any of the $V_{i, z}$ or $\frac{\g^i_1}{z-1}$ or their derivatives is involved, the pole order of the corresponding term is strictly less than $j$ in view of the estimate \er{order of V} where $\gamma_{0, i}>-1$ and the obvious reason. (Such an argument was also used in \cite{BC} for computing the exponents of the Cauchy-Euler equation with singularities only at 0 and $\infty$.)

Therefore when computing the exponents at 0, we can replace the $U_{i, z}$ in \er{my way} by $\frac{\gamma^i_0}{z}$, and the exponents of \er{ode} are the exponents of the following ODE 
\begin{equation}
\label{alpha}
\Big(\p_z - \frac{\gamma^{n-1}_{0}}{z}\Big) \Big(\p_z + \frac{\gamma^{n-1}_{0} - \gamma^{n-2}_{0}}{z}\Big) \cdots \Big(\p_z + \frac{\gamma^{2}_{0} - \gamma^{1}_{0}}{z}\Big) \Big(\p_z + \frac{\gamma^{1}_{0}}{z}\Big) h = 0
\end{equation}
for an unknown function $h(z)$. 
This is then 
$$
\frac{1}{z}\Big(\th - {\gamma^{n-1}_{0}}\Big) \frac{1}{z}\Big(\th + {\gamma^{n-1}_{0} - \gamma^{n-2}_{0}}\Big) \cdots \frac{1}{z}\Big(\th + {\gamma^{2}_{0} - \gamma^{1}_{0}}\Big) \frac{1}{z}\Big(\theta + {\gamma^{1}_{0}}\Big) h = 0.
$$
Using $(\theta+a)\frac{1}{z} = \frac{1}{z}(\theta+a-1)$, we see that this is 
$$
\frac{1}{z^{n}}\Big(\th - {\gamma^{n-1}_{0}} -(n-1)\Big) \Big(\th + {\gamma^{n-1}_{0} - \gamma^{n-2}_{0}}-(n-2)\Big) \cdots \Big(\th + {\gamma^{2}_{0} - \gamma^{1}_{0}}-1\Big) \Big(\theta + {\gamma^{1}_{0}}\Big) h = 0.
$$
Then the exponents are 
\begin{equation}\label{easier}
-\g_0^1, -\g_0^2 +\g_0^1+1, \cdots, -\g_0^{n-1} + \g_0^{n-2} + (n-2), \g_0^{n-1} + (n-1).
\end{equation}
The successive differences are seen to be $\mu_{0, i}$ from \er{mu}. 
Therefore the exponents of \er{ode} at 0 are, in an increasing order,  
\begin{equation}\label{exponents at 0}
-\gamma^{1}_{0}, -\gamma^{1}_{0} + \mu_{0, 1},  \cdots,  -\gamma^{1}_{0} + \mu_{0, 1} + \cdots + \mu_{0, n-1}. 
\end{equation}

Similarly, the exponents at $z=1$ are, in an increasing order,
\begin{equation}\label{exponents at 1}
-\gamma^{1}_{1}, -\gamma^{1}_{1} + \mu_{1, 1},  \cdots,  -\gamma^{1}_{1} + \mu_{1, 1} + \cdots + \mu_{1, n-1}. 
\end{equation}

The exponents at $z=\infty$ are computed by applying the transformation $w=\frac{1}{z}$ and the relation $\p_z = -w^2\p_w$. 
Let 
$$
\t U_i(w) = U_i \Big(\frac{1}{w}\Big) = U_i(z),\quad\text{and}\quad \t V_i(w) = \t U_i(w) - 2\g_\infty^i\log|w|.
$$
Then
$$
U_{i, z} = -w^2 \t U_{i, w} = -w^2\Big(\t V_{i, w} + \frac{\g_\infty^i}{w}\Big).
$$
When computing the exponents, by similar reason as above, we can disregard the terms $\t V_{i, w}$. 
Therefore the exponents of equation \er{my way} at $z=\infty$ are the exponents at $w=0$ of the equation 
\begin{equation*}
(-1)^n w^2\Big(\p_w - \frac{\gamma^{n}_{\infty}}{w}\Big) 
\cdots w^2 \Big(\p_w + \frac{\gamma^{2}_{\infty} - \gamma^{1}_{\infty}}{w}\Big) w^2\Big(\p_w + \frac{\gamma^{1}_{\infty}}{w}\Big) \t h = 0.
\end{equation*}
This is the same as 
$$
(-1)^n w^n\Big(w\p_w - \gamma^{n}_{\infty} + (n-1)\Big) 
\cdots \Big(w\p_w + \gamma^{2}_{\infty} - \gamma^{1}_{\infty} + 1\Big) \Big(w\p_w + {\gamma^{1}_{\infty}}\Big) \t h = 0.
$$
Therefore the exponents are 
\begin{equation}\label{easy at 8}
-\g_\infty^1, -\g_\infty^2 + \g_\infty^1 -1, \cdots, -\ginf^{n-1} + \ginf^{n-2} - (n-2),  \g_\infty^{n-1} - (n-1),
\end{equation}
which are again seen to be, in an increasing order, 
\begin{equation}\label{exponents at 8}
-\gamma^{1}_{\infty}, -\gamma^{1}_{ \infty} + \mu_{\infty, 1},  \cdots,  -\gamma^{1}_{ \infty} + \mu_{\infty, 1} + \cdots + \mu_{\infty, n-1}
\end{equation}
using the notation from \er{mu8}. 
\end{proof}

\begin{remark} Clearly the Fuchsian property and the local exponents in Theorem \ref{prop-exp} continue to hold for $SU(n)$ Toda systems with more singular points. 
\end{remark}

\begin{remark} The exponents of a Fuchsian equation satisfy the following Fuchs relation. Let $\rho_1(P), \dots, \rho_n(P)$ be the exponents at a singular point $P$, then 
\begin{equation}\label{Fuchs}
\sum_{P \text{ singular}} \left(\rho_1(P) + \cdots + \rho_n(P) - {n \choose 2}\right) = -2 {n\choose 2}.
\end{equation}
It is interesting to check this in our case. Indeed, by \er{easier} and \er{easy at 8}, 
we see that  
\begin{align*}
\rho_1(P) + \cdots + \rho_n(P) &= {n \choose 2}, \quad P = 0 \text{ or }1,\\
\rho_1(\infty) + \cdots + \rho_n(\infty) &= -{n \choose 2},
\end{align*}
\end{remark}

To be able to relate ODE \er{ode} with hypergeometric equations, we would like to shift  the local exponents at 1 and thus we consider $g(z)$ from \er{def gz} with some constant $\tau$. 
\begin{proposition} \begin{enumerate}
\item The $g(z)$ in \er{def gz} satisfies the following ODE
\begin{multline}\label{eq g}
\Big(\p_z - U_{n-1, z} - \frac{\tau}{z-1}\Big)\cdots
\Big(\p_z + U_{2, z} - U_{1, z} - \frac{\tau}{z-1}\Big) \Big(\p_z + U_{1, z} - \frac{\tau}{z-1}\Big) g = 0.
\end{multline}

\item The above ODE after expansion is the Fuchsian equation
\begin{equation}\label{still rational}
(\p_z^n + R_1(z) \p_z^{n-1} + \cdots + R_n(z)) g = 0,
\end{equation}
where 
\begin{equation}\label{the R}
R_j(z) = \sum_{k=0}^j \frac{B_{j, k}}{z^{j-k}(z-1)^{k}},
\end{equation}
with the $B_{j, k}$ constants.

\item The local exponents of this ODE at $1$ and $\infty$ are the shifts  by $\tau$ and $-\tau$ of the exponents of \er{my way} as in \er{new at 1} and \er{new at 8}. 
\end{enumerate}
\end{proposition}

\begin{proof}
From the Lebniz rule, we know that as composition of operators, 
\begin{equation}\label{just one}
\p_z \circ (z-1)^{-\tau} = (z-1)^{-\tau} \Big(\p_z - \frac{\tau}{z-1}\Big).
\end{equation}
Therefore by \er{def gz} and the ODE \er{my way} satisfied by $f$, we have 
\begin{align*}
 0&=(\p_z - U_{n-1, z}) \cdots (\p_z + U_{2, z} - U_{1, z}) (\p_z + U_{1, z}) ((z-1)^{-\tau}  g) \\
 &= (z-1)^{-\tau} \Big(\p_z - U_{n-1, z} - \frac{\tau}{z-1}\Big)\cdots\\
 &\qquad\qquad\quad\ \, \Big(\p_z + U_{2, z} - U_{1, z} - \frac{\tau}{z-1}\Big) \Big(\p_z + U_{1, z} - \frac{\tau}{z-1}\Big) g,
\end{align*}
which implies Eq. \er{eq g}. 

Equation \er{still rational} is similarly obtained from \er{ode} and \er{Wj now} using the iteration of \er{just one} as
\begin{equation}\label{iter}
\p_z^k \circ (z-1)^{-\tau} = (z-1)^{-\tau} \Big(\p_z - \frac{\tau}{z-1}\Big)^k,
\end{equation}
which implies \er{the R} (see \er{shift-su3} for an example).  

Clearly,
the local exponents of \er{eq g} at $z=1$ are added by $\tau$, while those at $z=\infty$ by $-\tau$. More specifically, its local exponents at 1 and $\infty$ are 
\begin{gather}\label{new at 1}
\tau-\gamma^{1}_{1},\ \tau-\gamma^{1}_{1} + \mu_{1, 1},\cdots,\  \tau-\gamma^{1}_{1} + \mu_{1, 1} + \cdots + \mu_{1, n-1},\\
\label{new at 8}
-\tau-\gamma^{1}_{\infty}, -\tau-\gamma^{1}_{ \infty} + \mu_{\infty, 1},  \cdots, -\tau -\gamma^{1}_{ \infty} + \mu_{\infty, 1} + \cdots + \mu_{\infty, n-1}.
\end{gather}
\end{proof} 

\begin{proof}[Proof of Theorem \ref{necessary}]
We call the exponents in \er{new at 1} by 
$$
\rho_i = \tau - \g^1_1 + \sum_{j=1}^i \mu_{1, j},\quad 0\leq i\leq n-1.
$$
Since 0 is an exponent of a hypergeometric equation at $z=1$, we see that $\tau\in \br$. 
Then the $\rho_i$ 
are in an strictly increasing order.
Comparing these with the exponents in \er{the c}, we see that the $\g$ other than the integers $0, 1, \dots, n-2$ can be one of $\rho_i$. 

 If $\g=\rho_i$ for $1\leq i\leq n-1$, then we need to take $\tau = \g^1_1$ so that $\rho_0 = 0$. Moreover, we have 
$ \mu_{1, 1} + \cdots + \mu_{1, j} = j$ for $1\leq j\leq i-1$. This implies that $\gamma_{1, k} = 0$ for $1\leq k\leq i-1$. If $i<n-1$, we furthermore have that $\mu_1 + \cdots  + \mu_{j+1} = j$ for $j\geq i$. This implies that $\gamma_{1, i} + \g_{1, i+1} = -1$ and $\g_{1, j}=0$ for $j>i+1$. Therefore we have shown cases in \er{l&f} and \er{cond0}. 

If $\g=\rho_0$, then we have $\rho_j = j-1$ for $j\geq 1$. Therefore we need to take $\tau = \gamma_1^1 - \mu_{1, 1}$, and all the $\gamma_j=0$ for $j\geq 2$. 
This is case \er{1neq0}. 
\end{proof}

\section{Existence of solutions}

In this section, we prove Theorem \ref{hpg exist}, which asserts the existence of solutions for certain Toda systems \er{toda} with three singularities at $0, 1$ and $\infty$. 

\begin{proof}[Proof of Theorem \ref{hpg exist}] We will construct the $U_i$ for $1\leq i\leq n-1$ as real-valued and single-valued functions on $\bc\setminus \{0, 1\}$, prove that they satisfy the equation \er{U in zz}, and show that they have the right strength of singularities at $0, 1$ and $\infty$ as in \er{toda2}. 

With the $\alpha$ and $\beta$ defined as in \er{abs}, we consider the hypergeometric equation \er{dab}. 

By a result from \cite{BH} stated here as Theorem \ref{pos-inv-Herm} and our interlacing assumption, there exists a positive-definite Hermitian form $\ch(\cdot, \cdot)$ on the vector space $V$ of local solutions around $z_0=\frac{1}{2}$ that is invariant under the 
monodromy representation \er{monod}. Note that $\ch$ can be replaced by $\lambda \ch$ for $\lambda>0$. 

We choose a basis $(\sigma_1, \cdots, \sigma_n)$ of such local solutions. For $\eta\in \pi_1(\bp^1\setminus V, z_0)$, define its associated matrix $M(\eta)$ under \er{monod} with respect to the above basis by 
\begin{equation}\label{def M}
\phi(\eta) (\sigma_j) = \sum_{i=1}^n M(\eta)_{ij} \sigma_i, \quad j=1,\cdots, n.
\end{equation}
Let $H$ be the Hermitian matrix representing $\ch(\cdot,\cdot)$ with 
 $H_{ij} = \ch(\sigma_i, \sigma_j)$ for $1\leq i,j\leq n$. By the invariance of $\ch$, 
 \begin{multline*}
 H_{ij} = \ch(\sigma_i, \sigma_j) = \ch(\phi(\eta)(\sigma_i), \phi(\eta)(\sigma_j)) \\
 = \ch\Big(\sum_{k=1}^n M(\eta)_{ki} \sigma_k, \sum_{l=1}^n M(\eta)_{lj}\sigma_l\Big) = \sum_{k,l=1}^n M(\eta)_{ki} H_{kl} \ol{M(\eta)_{lj}},
 \end{multline*}
so 
$$
M(\eta)^t H \ol{M(\eta)} = H, \quad \forall \eta\in \pi_1(\bp^1\setminus V, z_0).
$$

Therefore the positive-definite Hermitian matrix $H^{-1}$ satisfies
\begin{equation}\label{inv for H1}
\ol{M(\eta)} H^{-1} M(\eta)^t = H^{-1},\quad \forall \eta\in \pi_1(\bp^1\setminus V, z_0).
\end{equation}
We write $H^{-1} = (H^{ij})$. 

The local solutions $\sigma_i$ can be extended to be functions on the universal covering of $\bp^1\setminus V$. 
With $\bsg = (\sigma_1, \cdots, \sigma_n)^t$, we consider the function 
$$
\bsg^* H^{-1} \bsg = \sum_{i, j=1}^n H^{ij} \ol{\sigma_i} \sigma_j,
$$
where $\bsg^* = \ol \bsg^t$ is the conjugate transpose. 
Then this function is invariant under the deck transformation of $\pi_1(\bp^1\setminus V, z_0)$ on the universal cover of $\bp^1\setminus V$, since by \er{def M} and \er{inv for H1}
$$
\phi(\eta) \Big(\sum_{i, j=1}^n H^{ij} \ol{\sigma_i} \sigma_j\Big) = \sum_{i, j, k, l=1}^n H^{ij} \ol{M(\eta)_{ki}} M(\eta)_{lj} \ol{\sigma_k} \sigma_l = \sum_{k, l=1}^n H^{kl} \ol{\sigma_k} \sigma_l.
$$
Therefore it descends to a function on $\bp^1\setminus V = \bc\setminus \{0, 1\}$. 

Now define
\begin{equation*}\label{U1}
U_1(z) = -\log( |z-1|^{-2\tau} (\bsg^* H^{-1} \bsg) ), \quad z\in \bc\setminus \{0, 1\}.
\end{equation*}
We note that the logarithm is taken for positive numbers and we take the real-valued branch. 

Define the vector of functions on the universal covering of $\bp^1\setminus V$
\begin{equation}\label{def nu}
\bnu = (\nu_1, \dots, \nu_n) := (z-1)^{-\tau} \bsg.
\end{equation}
It is clear from the above that $\bnu^* H^{-1} \bnu$ is single-valued on $\bc\setminus \{0, 1\}$, and 
\begin{equation}\label{u1nu}
U_1 = -\log(\bnu^* H^{-1} \bnu).
\end{equation}

To be able to write out the other $U_i$ in terms of $U_1$, we denote by $W=W(\bnu)$ the Wronskian matrix of $\bnu$ with respect to $z$, that is, 
\begin{equation}\label{Wrons}
W = \begin{pmatrix}
\bnu &
\bnu' &
\cdots &
\bnu^{(n-1)}
\end{pmatrix}.
\end{equation}
Now consider the Hermitian matrix 
\begin{equation}\label{big R}
R :=  {W}^* H^{-1} {W},
\end{equation}
where $W^* = \ol W^t$.  
Then \er{u1nu} says that 
$e^{- U_1} = R_{1,1}.$

In general, for a matrix $M$ of dimension $n$ and for $1\leq m\leq n$, 
let $M_m = M_{\ul{m}, \ul{m}}$ denote the leading principal minor of $M$ of dimension $m$, where $\ul{m}$ denotes the set $\{1, \cdots, m\}$. 

Because $H^{-1}$ is positive-definite, so is $R$. Therefore, $R_m>0$. 
It is standard (see, e.g., \cites{LS-book, N1, LWY}) that 
\begin{equation}\label{Jacobi}
R_{m}\big(\p_z\p_{\bar z} R_{m}\big) - \big(\p_z R_{m}\big)\big(\p_{\bar z} R_{m}\big)= R_{m-1} R_{m+1},\quad 1\leq m\leq n-1,
\end{equation}
where we define $R_{0} = 1$. 
By induction, we see that the $R_m$ are single-valued on $\bc\setminus \{0, 1\}$. 

Therefore defining 
\begin{equation}\label{Um}
U_m = -\log R_m, \quad 1\leq m\leq n-1,
\end{equation}
we see that the $U_m$ satisfy \er{U in zz} for $1\leq m\leq n-2$ with the Cartan matrix from \er{Cartan-A}. It is also clear that all solutions to \er{U in zz} with the given $U_1$ in \er{u1nu} are determined in this way. 

Now we show that \er{U in zz} for $m=n-1$ is satisfied if we replace $H$ by $\lambda H$ with a suitable scalar $\lambda>0$. For this, we will show that $R_{n} 
 = 1$ for a suitable $\lambda$. 

First, we show that $\det W(\bsg) = c(z-1)^{n\tau}$, where $c\neq 0\in \bc$. The hypergeometric equation \er{dab} has the following expansion 
{\allowdisplaybreaks
\begin{align*}
& (\theta+\beta_1-1)\cdots(\theta+\beta_n-1) - z(\theta + \a_1)\cdots(\theta + \a_n)\\
&= \bigg(\theta^n + \Big(\sum_{i=1}^n \beta_i - n\Big)\theta^{n-1} + \cdots\bigg) - z\bigg(\theta^n + \Big(\sum_{i=1}^n \a_i\Big)\theta^{n-1} + \cdots\bigg)\\
&= (1-z) \bigg(\theta^n + \Big(\sum_{i=1}^n \a_i + \frac{\sum_{i=1}^n \a_i - \sum_{i=1}^n \beta_i + n}{z-1}\Big) \theta^{n-1} + \cdots\bigg)\\
&\overset{(*)}= (1-z) \bigg(\theta^n + \Big(-n\tau - \frac{n(n-1)}{2} - \frac{n\tau}{z-1}\Big) \theta^{n-1} + \cdots\bigg)\\
&\overset{(**)}= (1-z) \bigg(z^n\frac{\p^n}{\p z^n} - \Big(\frac{n\tau z}{z-1}\Big) z^{n-1}\frac{\p^{n-1}}{\p z^{n-1}} + \cdots\bigg)\\
&= (1-z)z^n \bigg(\frac{\p^n}{\p z^n} - \Big(\frac{n\tau}{z-1}\Big) \frac{\p^{n-1}}{\p z^{n-1}} + \cdots\bigg),
\end{align*}
}
where dots stand for linear differential operators of order $\leq n-2$. 
Here equality $(*)$ uses the definition \er{abs} of the $\a$ and $\beta$, and the fact that 
$$
n\g_P^1 = (n-1)\g_{P, 1} + \cdots + \g_{P, n-1}, \quad P\in V
$$
from \er{def alpha} and the inverse of \er{Cartan-A}. 
Equality $(**)$ uses the identity 
$$
\theta^n = z^n \frac{\p^n}{\p z^n} + \frac{n(n-1)}{2} z^{n-1}\frac{\p^{n-1}}{\p z^{n-1}} + \cdots,
$$
which can be proved easily by induction. 
Therefore $\det W(\bsg)=c(z-1)^{n\tau}$ by Abel's theorem. 

By \er{def nu}, $\det W(\bnu)= (z-1)^{-n\tau} \det W(\bsg) = c$ is a nonzero constant. We can chose the multiple of $H$ so that 
$$
R_n = \det R = (\det (\lambda H)^{-1}) |\det W(\bnu)|^2 = 1. 
$$

Now we show that the solutions $U_m$ in \er{Um} have the right strength of singularities at the three singular points. By \er{ghg 0} and \er{abs}, the exponents of the hypergeometric equation for $\bsg$ at $z=0$ are, in a strictly increasing order, given in \er{exponents at 0} or equivalently \er{easier}. 
Under the interlacing condition, they are distinct modulo $\bz$. 
By the standard theory of ODEs (see \cite{Ince}) and \er{def nu}, there exists a basis of formal functions near 0 
$$
{\bzt^0} = (z^{-\g_0^1}\zeta_1^0, z^{-\g_0^2 + \g_0^1 + 1}\zeta_2^0, \dots, z^{-\g^{n-1}_0 + \g^{n-2}_0 + n-2} \zeta_{n-1}^0, z^{\g^{n-1}_0 + n-1} \zeta_n^0)^t,
$$
where the $\zeta_i^0$ are analytic functions at 0 and $\zeta_i^0(0)\neq 0$, such that
$$
\bnu = A^0 \bzt^0,
$$
for some  invertible constant matrix $A^0$. 

A standard calculation shows that the Wronskian of a sequence of power functions is 
$$
\det W(z^{\kappa_1}, \cdots, z^{\kappa_m}) = z^{\sum_{i=1}^m \kappa_i - \frac{m(m-1)}{2}} \prod_{i<j} (\kappa_j - \kappa_i),
$$
where the $\kappa_i$ are the exponents. 
Therefore for $1\leq m\leq n-1$ and from \er{big R}, 
\begin{align*}
R_m &= (W(\bzt^0)^* (A^0)^* H^{-1} A^0 W(\bzt^0))_m 
= ((A^0)^* H^{-1} A^0)_m |W(\bzt^0)_m|^2 (1 + o(1)) \\
&= |z|^{-2\g_0^m}(c_m^0 + o(1)) ,\quad \text{as }z\to 0,\text{ where }
c_m^0 >0.
\end{align*}
In view of \er{Um}, $U_m$ has the required strength of singularity at $z=0$ in \er{toda2}.

Corresponding to $z=\infty$, we use $w=\frac{1}{z}$ as the local coordinate and consider $w=0$. 
Then as above, \er{ghg 8}, \er{abs}, \er{easy at 8} and \er{def nu} give $\bnu = A^\infty \bzt^\infty$ for some invertible matrix $A^\infty$ and 
$$
{\bzt^\infty} = (w^{-\g_\infty^1}\zeta^\infty_1, w^{-\g_\infty^2 + \g_\infty^1 + 1}\zeta^\infty_2, \dots, w^{-\g^{n-1}_\infty + \g^{n-2}_\infty + n-2} \zeta^\infty_{n-1}, w^{\g^{n-1}_\infty + n-1} \zeta^\infty_n)^t,
$$
where the $\zeta^\infty_i$ are analytic functions in $w$ at $w=0$ with nonzero constant terms. 
The same arguments as above apply and give 
$$
e^{-U_m} = |w|^{-2\g_\infty^m} (c_m^\infty + o(1)) = |z|^{2\g_\infty^m} (c_m^\infty + o(1)), \quad \text{as }w\to 0\text{ or }z\to \infty,
$$
which shows the required singularity at $\infty$ in \er{toda2}. 

At last, we consider the point $z=1$. First we show that under the interlacing condition in Theorem \ref{hpg exist}, the $\g=\sum_{j=1}^n \beta_j - \sum_{i=1}^n \a_i -1$ in \er{the c} is not an integer. 
Indeed, mod $\bz$, we see $0<\g - \lf \g\rf <1$ by the interlacing condition. Therefore by the standard ODE theory, there exists a local formal solution 
$(z-1)^\g \zeta^1_{n}$ around $z=1$ with $\zeta^1_n$ analytic with nonzero constant term. 

By \cite{BH}*{Proposition 2.8}, we see that the hypergeometric equation \er{dab} has $n-1$ analytic solutions near $z=1$ of the form
\begin{align*}
(z-1)^{j-1}\zeta^1_j(z) &= (z-1)^{j-1} + O((z-1)^{n-1}),\quad j=1, \dots, n-1,
\end{align*}
corresponding  to the exponents $0, 1, \dots, n-2$ in \er{the c}. 
Therefore despite that the exponents of \er{dab} at 1 are not distinct mod $\bz$, we still have a basis of formal solutions with the corresponding exponents. 

By \er{def nu} and the proof of the necessary conditions in Theorem \ref{necessary}, we see that around $z=1$, $\bnu = A^1 \bzt^1$ with $A^1$ invertible and 
\begin{multline*}
{\bzt^1} = ((z-1)^{-\g_1^1}\zeta^1_1, (z-1)^{-\g_1^2 + \g_1^1 + 1}\zeta^1_2, \cdots,\\ (z-1)^{-\g^{n-1}_1 + \g^{n-2}_1 + n-2} \zeta^1_{n-1}, (z-1)^{\g^{n-1}_1 + n-1} \zeta^1_n)^t,
\end{multline*}
where $\zeta^1_i$ are analytic and $\zeta^1_i(1)\neq 0$. 
Then a similar analysis as above applies.  
\end{proof}

\section{Sufficiency for $SU(3)$ Toda systems by Pohozaev identity}\label{sect-su3}

We note again that for a Fuchsian equation of order $n\geq 3$ with singularities at $0$, $1$ and $\infty$, having just the right exponents at $z=1$ as in \er{the c} doesn't guarantee that it is actually a hypergeometric equation. 
Indeed, a general equation as \er{gen ode} is hypergeometric if it is Fuchsian with regular singularities at 0, 1, and $\infty$, and furthermore all the coefficients $a_i(z)$ have \emph{simple} poles at $z=1$ \cite{BH}. 
It would be ideal if the {necessary conditions} in Theorem \ref{necessary} are also {sufficient} for equation \er{eq g} to be hypergeometric if the $U_i$ satisfy a Toda system. 
In this section, we prove Theorem \ref{thm-su3} for the sufficiency for $SU(3)$ Toda system by a Pohozaev identity under some mild reality assumption. 
We would also like to call attention to the relationship of our current results, in terms of estimates and integrals, with those of 
\cite{KMPS} where higher dimensional equations related to the scalar curvature are studied. 

For simplicity and concreteness, we write the Toda system \er{toda2} for $SU(3)$ as
\begin{equation}\label{su3}
\begin{cases}
\Delta u + 4 e^{2u-v} = 4\pi (\gamma_0^u \delta_0 + \gamma_1^u \delta_1), 
\\
\Delta v + 4 e^{2v-u} = 4\pi (\gamma_0^v \delta_0 + \gamma_1^v \delta_1), 
\\
u(z) = -2\gamma_\infty^u \log|z| + O(1)\quad \text{near }\infty, 
\\
v(z) = -2\gamma_\infty^v \log|z| + O(1)\quad \text{near }\infty, 
\end{cases}
\end{equation}
where $u=U_1$, $v=U_2$, $\gamma^u_P = \gamma^1_P$, and $\gamma^v_P = \gamma^2_P$. 
We recall that
\begin{equation}\label{gmgm}
\gamma_{1, 1} = 2\gamma_1^u - \g_1^v>-1,\quad \g_{1, 2}=2\g_1^v - \g_1^u>-1.
\end{equation} 
We have from \er{integ} that 
\begin{equation}\label{su3int}
\begin{split}
4\int_{\br^2} e^{2u-v} &= 4\pi(\gamma_0^u + \gamma_1^u + \gamma_\infty^u),\\
4\int_{\br^2} e^{2v-u} &= 4\pi(\gamma_0^v + \gamma_1^v + \gamma_\infty^v).
\end{split}
\end{equation}

Our goal is to use the Pohozaev identity to study the ODE operator \er{my way} for $f=e^{-u}$: 
\begin{equation}\label{eqsu3}
\begin{split}
&(\p_z - v_{z}) (\p_z + v_{z} - u_{z}) (\p_z + u_{z}) = \p_z^3 + W_2\p_z + W_3\\
&= \p_z^3 + \Big(\frac{A}{z^2} + \frac{B}{z(z-1)} + \frac{C}{(z-1)^2}\Big)\p_z + \Big(\frac{D}{z^3} + \frac{E}{z^2(z-1)}
+ \frac{F}{z(z-1)^2} + \frac{G}{(z-1)^3}\Big),
\end{split}
\end{equation}
where we have used Proposition \ref{prop-wj} and we have used shorthands for the coefficients $A_{j, k}$. 
We first want to determine these seven coefficients in terms of the six strength of singularities in \er{su3} as much as possible. 
The following formulas are equivalent to the exponent information in \er{exponents at 0}, \er{exponents at 1}, and \er{exponents at 8}. 
\begin{proposition}\label{6of7}
 We have the following formulas
\begin{align}
C &= -\g_1^u - \g_1^v - (\g_1^u)^2 - (\g_1^v)^2 + \g_1^u \g_1^v,\label{for AC}\\
G &= 2\g_1^u + 2(\g_1^u)^2 - \g_1^u\g_1^v + (\g_1^u)^2\g_1^v - \g_1^u(\g_1^v)^2,\label{the G}
\end{align}
and similar formulas for $A$ and $D$ with the $\g_1$'s replaced by the $\g_0$'s.
Furthermore, we have that 
\begin{align}
A + B + C + 2 &= \rho_1\rho_2 + \rho_1 \rho_3 + \rho_2 \rho_3,\label{for B}\\
D + E + F + G &= \rho_1\rho_2\rho_3,\label{EF}
\end{align}
where 
\begin{equation}\label{the rho8}
\rho_1 = -\ginf^u,\ \rho_2 = \ginf^u - \ginf^v - 1,\ \rho_3 = \ginf^v -2.
\end{equation} 
\end{proposition}

\begin{proof}
The coefficients $C$ and $G$ are for the highest order poles at $z=1$ in $W_2$ and $W_3$ in \er{eqsu3}, and they are equivalent to the information about the exponents at 1 in \er{exponents at 1}. In fact, the $C$ and $G$ are computed by expanding
$$
\Big(\p_z - \frac{\g_1^v}{z-1}\Big) \Big(\p_z + \frac{\g_1^v-\g_1^u}{z-1}\Big) \Big(\p_z + \frac{\g_1^u}{z-1}\Big).
$$
Similarly we determine $A$ and $D$ at $z=0$. 

Now we consider the exponents at $\infty$. By \er{easy at 8}, they are exactly the $\rho_i$ in \er{the rho8}. 
By the standard way of computing the exponents at $\infty$ (see \cite{Ince}), these exponents are the roots of the equation
$$
-s(s+1)(s+2) - (A+B+C)s + D+E+F+G = 0,
$$
which is 
$$
-(s^3 + 3s^2 + (A+B+C+2) s - (D+E+F+G)) = 0.
$$
Therefore the standard relationship between roots and coefficients gives \er{for B} and \er{EF}. 
\end{proof}

Now we introduce the Pohozaev identity in our situation. 
We denote the standard dot product by $\la\cdot, \cdot\ra$ and use the standard notation that $x=(x_1, x_2)$ and $\nabla u = (\p_{x_1} u, \p_{x_2} u)$. 

\begin{lemma} Let $\Omega\subset \bc\setminus \{0, 1\}$ be a domain with $C^1$ boundary. 
We have the following Pohozaev identity
\begin{equation}\label{pz}
\begin{split}
8\int_\Omega e^{2u-v} + e^{2v-u} &= \int_{\p \Omega} \,2\la x, \nabla u\ra \la \nu, \nabla u\ra - \la x, \nu\ra |\nabla u|^2\\
& + 2\la x, \nabla v\ra \la \nu, \nabla v\ra - \la x, \nu\ra |\nabla v|^2\\
&-  \la x, \nabla v\ra \la \nu, \nabla u\ra - \la x, \nabla u \ra \la \nu, \nabla v\ra\\
&+ \la x, \nu\ra \la\nabla u, \nabla v\ra + 4\la x, \nu\ra (e^{2u-v} + e^{2v-u}),
\end{split}
\end{equation}
where $\nu$ stands for the outward unit normal vector of $\p \Omega$. 
\end{lemma}

\begin{proof} The proof is by standard manipulations. For concreteness, we provide some details. 
Multiply the first equation in \er{su3} by $\la x, 2\nabla u - \nabla v\ra$ and the second by $\la x, 2\nabla v - \nabla u\ra$ and add them, then we get in $\Omega$ 
\begin{align*}
0&= \Delta u \la x, 2\nabla u - \nabla v\ra + \Delta v \la x, 2\nabla v - \nabla u\ra + 4\la x, \nabla e^{2u-v} + \nabla e^{2v-u}\ra \\
&= \on{div}(\la x, 2\nabla u - \nabla v\ra \nabla u) + \on{div}(\la x, 2\nabla v - \nabla u\ra \nabla v) \\
&- \on{div}((|\nabla u|^2 - \la \nabla u, \nabla v\ra + |\nabla v|^2) x)\\
&+ 4\on{div}((e^{2u-v} + e^{2v-u}) x) - 8(e^{2u-v} + e^{2v-u}).
\end{align*}
Then we apply the divergence theorem. 
\end{proof}

\begin{lemma} We have
\begin{equation}\label{itie}
\begin{split}
8\int_{\br^2} e^{2u-v} + e^{2v-u} &= 8\pi(((\gamma_\infty^{u})^2 + (\gamma_\infty^{v})^2 - \gamma_\infty^{u}  \gamma_\infty^{u}) \\
& - ((\gamma_0^{u})^2 + (\gamma_0^{v})^2 - \gamma_0^{u}  \gamma_0^{u})\\
& - ((\gamma_1^{u})^2 + (\gamma_1^{v})^2 - \gamma_1^{u}  \gamma_1^{u})) - I_e,
\end{split}
\end{equation}
where $e=(1, 0)\in \br^2$ (corresponding to $1\in \bc$) and 
\begin{equation}\label{ie}
\begin{split}
I_e &= \lim_{r\to 0} \int_{S_r(1)} 2\la e, \nabla u\ra \la \nu, \nabla u\ra - \la e, \nu\ra |\nabla u|^2\\
&+ 2\la e, \nabla v\ra \la \nu, \nabla v\ra - \la e, \nu\ra |\nabla v|^2\\
&- \la e, \nabla v\ra \la \nu, \nabla u\ra - \la e, \nabla u \ra \la \nu, \nabla v\ra\\
&+ \la e, \nu\ra \la\nabla u, \nabla v\ra + 4\la e, \nu\ra (e^{2u-v} + e^{2v-u})\, ds.
\end{split}
\end{equation}
\end{lemma}

\begin{proof} We apply \er{pz} to the domain $\Omega = B_R(0) \backslash (B_r(0) \cup B_r(1))$, and take the limit as $R\to \infty$ and $r\to 0$. Then we get 
$$
8\int_{\br^2} e^{2u-v} + e^{2v-u} = I_\infty - I_0 - I_1,
$$
where $I_\infty$, $I_0$ and $I_1$ are respectively limits of integrals over $S_R(0)$, $S_r(0)$ and $S_r(1)$. 
We first consider the integral $I_0$. Then $|z| = r$, and $x = r \nu$. Note that by \er{def V} and \er{order of V} we have 
\begin{equation}\label{easy est}
\begin{split}
u &= 2\g_0^u \log r + O(1),\\
\nabla u &= \frac{2\g_0^u}{r}\nu + o(r^{-1}), \quad \text{as }r\to 0,
\end{split}
\end{equation}
and similarly for $v$. Therefore we have 
\begin{align*}
I_0 &= \lim_{r\to 0}\int_0^{2\pi} r^2( ((2\g_0^u)^2 + (2\g_0^v)^2 - (2\g_0^u)(2\g_0^v)) r^{-2}\\
&\qquad\qquad\quad + o(r^{-2}) + O(r^{2\g_{0, 1}}) + O(r^{2\g_{0, 2}}))\, d\theta\\
&= 8\pi((\g_0^{u})^2 + (\g_0^{v})^2 - \g_0^u \g_0^v),
\end{align*}
by \er{gmgm}. 

We apply similar estimates to $S_R(0)$ and $S_r(1)$ where we split $x = (x-e) + e$. Therefore we are done with one new integral $I_e$ as above.   
\end{proof}

To be able to evaluate $I_e$, we need finer asymptotics (at 1) about the solutions than \er{easy est}. 
We thank Wen Yang for his help with the following lemma and proposition.
\begin{lemma}\label{lem-radial}
Let $\kappa>-1$. For $z \in B_{\frac{1}{2}}(1)\subset \bc$, 
$$
-\frac{1}{2\pi}\int_{B_{\frac{1}{2}(1)}} \log|z-w| |w-1|^{2\kappa}\, dA = a |z-1|^{2\kappa+2} + b,
$$
for some constants $a$ and $b$, where $dA=\frac{i}{2}dw\wedge d\bar w$ is the area element. 
\end{lemma}

\begin{proof} We first note that the integral on the LHS is convergent for all $z\in B_{\frac{1}{2}}(1)$, and we denote the corresponding function by $f(z)$. 
Let $r=|z-1|$. We then note that $f(z) = f(r)$ is a radial function about $1$ by the symmetry of $|w-1|^{2\kappa}$ and the invariance of the Lebesgue measure under rotation. 
By standard potential analysis, we see that 
$$
f(x) = f(r) = ar^{2\kappa + 2} + b,
$$
where $a$ can be computed as $-\frac{1}{(2\kappa + 2)^2}$. 
\end{proof}



\begin{proposition}\label{finer}
For solutions $u$ and $v$ of \er{su3}, we have, with  $r=|z-1|$, that
\begin{equation}\label{asymp}
\begin{split}
u(z) &= 2\gamma_1^u\log r + p(r) + 
c(x_1-1) + \t c x_2+ o(r),\\
v(z) &= 2\gamma_1^v\log r + q(r) + 
d(x_1-1) + \t d x_2 + o(r), \quad \text{as }z\to 1,
\end{split}
\end{equation}
where $p$ and $q$ are functions of $r$ and we have 
\begin{equation}\label{radial part}
\begin{split}
p(r) &= O(1), \quad p'(r) = o(r^{-1}),\\
q(r) &= O(1), \quad q'(r) = o(r^{-1}),\quad \text{as }r\to 0.
\end{split}
\end{equation}
\end{proposition}

\begin{proof} Let 
\begin{gather*}
U = u - 2\g_1^u \log |z-1|,\\
V = v - 2\g_1^v \log |z-1|.
\end{gather*}
Then $U$ and $V$ satisfy
$$
\begin{cases}
\Delta U = -4|z-1|^{2\g_{1, 1}} e^{2U-V}\\
\Delta V = -4|z-1|^{2\g_{1, 2}} e^{2V-U},\quad \text{in }B_{\frac{1}{2}}(1).
\end{cases}
$$

Let $\beta = 2(\min(\g_{1, 1}, \g_{1, 2}) + 1)>0$.
By \cites{BM, CL0}, we have the following integral representation 
\begin{align}
U(z) &= -\frac{4}{2\pi}\int_{B_{\frac{1}{2}}(1)} \log|z-w| |w-1|^{2\g_{1, 1}} e^{2U(w)-V(w)}\, dA + C^1,\label{UU}\\
V(z) &= -\frac{4}{2\pi}\int_{B_{\frac{1}{2}}(1)} \log|z-w| |w-1|^{2\g_{1, 2}} e^{2V(w)-U(w)}\, dA + C^1,\nm
\end{align}
where $C^1$ stands for $C^1$ functions at point $z=1$. Furthermore, by the above works (cf. \er{important}) we know that
$$
U, V\in C^{0, \alpha}(B_{\frac{1}{2}}(1))
$$
for some positive constant $0<\alpha<\beta$. Write 
\begin{gather*}
U(z) = U(1) + O(r^\a),\\
V(z) = V(1) + O(r^\a).
\end{gather*}

Plugging these in \er{UU} and using the above lemma, we get 
\begin{align*}
U(z) &= -\frac{4}{2\pi}\int_{B_{\frac{1}{2}}(1)} \log|z-w| |w-1|^{2\g_{1, 1}} e^{2U(1)-V(1)}( 1 + O(r^\a))\, dA + C^1\\
&= U(1) + a_1 r^{2\g_{1, 1} + 2} + O(r^{\min(\a+\beta, 1)}).
\end{align*}
Similarly, we have
$$
V(z) =  V(1) + b_1 r^{2\g_{1, 2} + 2} + O(r^{\min(\a+\beta, 1)}).
$$

Plugging these back in \er{UU}, we get 
\begin{align*}
U(z) &= -\frac{4}{2\pi}\int_{B_{\frac{1}{2}}(1)} \log|z-w| |w-1|^{2\g_{1, 1}} e^{2U(1)-V(1)}\\
& \qquad ( 1 + 2a_1 r^{2\g_{1, 1} + 2} - b_1 r^{2\g_{1, 2}+2} +  O(r^{\a + \beta}))\, dA + C^1\\
&= U(1) + a_1 r^{2\g_{1, 1} + 2} + a_{2, 1} r^{2(2\g_{1, 1} + 2)} + a_{2, 2} r^{2\g_{1, 1} + 2\g_{1, 2} + 4} + O(r^{\min(\a+2\beta, 1)}),
\end{align*}
and similarly for $V(z)$. 

Continuing this way and expanding the exponential function to higher powers as needed, 
we see that all the powers with exponents less than 1 are radial because $\a + n\beta$ is eventually bigger than 1. 
We take the functions $p(r)$ and $q(r)$ to be the corresponding sums of such powers. 
\end{proof}

\begin{proposition} 
The $I_e$ in \er{ie} evaluates to 
\begin{equation}\label{ansie}
I_e = 4\pi(c\gamma_{1, 1} + d\gamma_{1, 2}),
\end{equation}
where $c$ and $d$ are from \er{asymp}. 
\end{proposition}

\begin{proof}
Let $\nu$ be the outward unit normal of $S_r(1)$, and let $\t e = (0, 1)$. 
 Using polar coordinates with $x_1=1+r\cos\theta$ and $x_2=\sin\theta$, we have $\la e, \nu\ra = \cos\theta$ and $\la \t e, \nu\ra = \sin\theta$. 

We first show that the last integral in \er{ie} converges to zero:
\begin{align*}
\int_{S_r(1)} \la e, \nu\ra e^{2u-v}\, ds &= \int_0^{2\pi} r (\cos\theta) \exp(2\g_{1, 1}\log r + (2p-q) + O(r))\, d\theta \\
&= \int_0^{2\pi} r (\cos\theta) r^{2\gamma_{1, 1}} e^{2p-q}  e^{O(r)}\, d\theta\\
&= \int_0^{2\pi} r (\cos\theta) r^{2\gamma_{1, 1}}  e^{2p-q}  {(1+O(r))})\, d\theta\\
&= 0 + \int_0^{2\pi} O(r^{2+2\gamma_{1, 1}})\, d\theta\\
&\to 0,\quad \text{as }r\to 0,
\end{align*}
since $\gamma_{1, 1}>-1$. Similarly the other integral for $e^{2v-u}$ also goes to zero. 

From \er{asymp}, we see that  as $r\to 0$, 
\begin{equation}
\label{uvexp}
\begin{split}
\nabla u &= \frac{2\gamma_1^u }{r} \nu + p'(r) \nu + c e + \t c \t e+ o(1),\\
\nabla v &= \frac{2\gamma_1^v }{r} \nu + q'(r) \nu + d e + \t d \t e + o(1).
\end{split}
\end{equation}
Now we compute the integrals of the first two terms in \er{ie} and the others are similar. For the first term, by \er{radial part} we have 
\begin{align*}
&\int_{S_r(1)} 2\la e, \nabla u\ra \la \nu, \nabla u\ra\, ds \\
&= 2\int_0^{2\pi} r\Big(\frac{2\g_1^u}{r}\cos\theta + p'(r)\cos\theta + c + o(1)\Big)
\Big(\frac{2\g_1^u}{r} + p'(r) + c\cos\theta + \t c\sin\theta + o(1)\Big)\, d\theta\\
&= 2\int_0^{2\pi} r\Big(P_1(r)\cos\theta + P_2(r)\cos\theta \sin\theta+ P_3(r)\sin\theta + \frac{2c\g_1^u}{r}\cos^2\theta +  \frac{2c\g_1^u}{r} + o(r^{-1})\Big)\, d\theta\\
&\to 2\pi\cdot 2\cdot (3c\g_1^u),\quad \text{as }r\to 0.
\end{align*}
Here the $P_i(r)$ stand for the function of $r$ after expansion. For the second term, we have 
\begin{align*}
&-\int_{S_r(1)} \la e, \nu\ra |\nabla u|^2\, ds\\
&= -\int_0^{2\pi} r\cos\theta\Big( \big(\frac{2\g_1^u}{r} + p'(r)\big)^2 + c^2 + \t c^2 +  2\big(\frac{2\g_1^u}{r} + p'(r)\big)(c\cos\theta + \t c\sin\theta)+ o(r^{-1})\Big)\, d\theta\\
&= -\int_0^{2\pi} r\Big( P(r)\cos\theta + Q(r)\cos\theta \sin\theta + \frac{4c\g_1^u}{r}\cos^2\theta + o(r^{-1})\Big)\, d\theta\\
&\to -2\pi( 2c\g_1^u), \quad \text{as }r\to 0.
\end{align*}

Repeating this for the other terms, we see that 
\begin{align*}
I_e &= 2\pi( 6c\g_1^u - 2c\g_1^u + 6d\g_1^v - 2d\g_1^v - (2d\g_1^u + c\g_1^v) - (2c\g_1^v + d\g_1^u) + (d\g_1^u + c\g_1^v))\\
&= 4\pi( c(2\g_1^u - \g_1^v) + d(2\g_1^v - \g_1^u)) = 4\pi( c\g_{1, 1} + d\g_{1, 2}).
\end{align*}
\end{proof}

Furthermore, using the asymptotic expansion \er{asymp}, we can determine $F$ in \er{eqsu3} and thus determine all the coefficients in view of Proposition \ref{6of7}.  
\begin{proposition} The real and imaginary parts of the coefficient $F$ in \er{eqsu3} are  
\begin{align}\label{for F}
\begin{split}
\on{Re} F &= \frac{1}{2}( 2c\g_1^u - c\g_1^v - c(\g_1^{v})^2 - 2d\g_1^u\g_1^v + d(\g_1^u)^2 + 2c\g_1^u\g_1^v  )\\
&= \frac{1}{2} (c(\g_1^v + 1) \g_{1, 1} - d\g_1^u \g_{1, 2}), \\
\on{Im} F &= -\frac{1}{2} (\t c(\g_1^v + 1) \g_{1, 1} - \t d\g_1^u \g_{1, 2}).
\end{split}
\end{align}
\end{proposition}

\begin{proof} By direct computation of \er{eqsu3}, we have 
\begin{equation}\label{explicit W}
W_3 = u_{zzz} - 2u_{zz} u_z + v_{zz} u_z - u_z v_z^2 + u_z^2 v_z.
\end{equation}
This is equal to $\frac{D}{z^3} + \frac{E}{z^2(z-1)}
+ \frac{F}{z(z-1)^2} + \frac{G}{(z-1)^3}$, whose Laurent expansion at $z=1$ is 
$$
 \frac{G}{(z-1)^3} + \frac{F}{(z-1)^2} + \frac{-F + E}{z-1} + (F - 2E + D) + O(|z-1|).
$$
Therefore $F$ is the coefficient of $(z-1)^{-2}$ in \er{explicit W}. 

Using the asymptotic expansion \er{asymp}, we see 
\begin{align*}
u_z &= \frac{\gamma_1^u}{z-1} + \frac{c}{2} - \frac{\t c}{2}i + \p_z p(r) + o(1),\\
u_{zz} &= -\frac{\gamma_1^u}{(z-1)^2} + \p_{z}^2 p(r) + o(r^{-1}),\\
u_{zzz} &= 2\frac{\gamma_1^u}{(z-1)^3} + \p_z^3 p(r) + o(r^{-2}),\\
v_z &= \frac{\gamma_1^v}{z-1} + \frac{d}{2} - \frac{\t d}{2}i + \p_z q(r) + o(1).
\end{align*}

In computing the term containing $(z-1)^{-2}$, we now argue that the error terms and the derivatives of $p$ or $q$ can be neglected. 
The $o(1)$, $o(r^{-1})$ and $o(r^{-2})$ can be omitted by the reason of pole orders, since they are little $o$ of a term whose pole order is one less than the corresponding highest pole order. 
The derivatives of $p$ and $q$ can be omitted since $W_3$ is meromorphic but $p$ and $q$ are functions of $r = |z-1|$. 
Therefore, direct computation gives $F$ as in \er{for F}, 
noting that the second line is a simplification of the first in view of \er{gmgm}. (We remark that $G$ can be computed as the coefficient of $(z-1)^{-3}$ in \er{explicit W} and this recovers \er{the G}.) 
\end{proof}

\begin{proof}[Proof of Theorem \ref{thm-su3}] We first compute equation \er{eq g} in the form \er{still rational} in the current case using \er{iter}, and the result is 
{\allowdisplaybreaks
\begin{align}
&\Big(\p_z - \frac{\tau}{z-1}\Big)^3 + \Big(\frac{A}{z^2} + \frac{B}{z(z-1)} + \frac{C}{(z-1)^2}\Big)\Big(\p_z - \frac{\tau}{z-1}\Big) \nm\\
&+ \Big(\frac{D}{z^3} + \frac{E}{z^2(z-1)}
+ \frac{F}{z(z-1)^2} + \frac{G}{(z-1)^3}\Big)\nm\\
&\quad=\p_z^3 - \frac{3\tau}{z-1}\p_z^2 + \Big(\frac{A}{z^2} + \frac{B}{z(z-1)} + \frac{C+3\tau^2+ 3\tau}{(z-1)^2}\Big)\p_z\nm \\
&\qquad+ \Big(\frac{D}{z^3} + \frac{E-A\tau}{z^2(z-1)}
+ \frac{F-B\tau}{z(z-1)^2} + \frac{G-C\tau-\tau^3-3\tau^2-2\tau}{(z-1)^3}\Big) \label{shift-su3}\\
&\quad=: \p_z^3 - \frac{3\tau}{z-1}\p_z^2 + \Big(\frac{A'}{z^2} + \frac{B'}{z(z-1)} + \frac{C'}{(z-1)^2}\Big)\p_z\nm \\
&\qquad+ \Big(\frac{D'}{z^3} + \frac{E'}{z^2(z-1)}
+ \frac{F'}{z(z-1)^2} + \frac{G'}{(z-1)^3}\Big),\nm
\end{align}
}
where the new coefficients are denoted with primes. To prove that this equation is hypergeometric, we need to show that the coefficients have simple poles at $z=1$, that is, 
$C'=G'=F'=0$. 

The necessary conditions in Theorem \ref{necessary} (2) in the $SU(3)$ cases are 
\begin{equation}\label{su3 cases}
0)\ \g_{1, 2}=0,\quad \text{or} \quad 1)\ \g_{1, 1} + \g_{1, 2} = -1,\quad \text{or} \quad 2)\ \g_{1, 1}=0
\end{equation}
Let us study cases 1) and 2) first. By \er{l&f} and \er{cond0}, we see $\tau = \g_1^u$. Then it can be computed that $C'=0$ and $G'=0$ in both cases. We remark that this can be seen using just the exponents of \er{shift-su3} at $z=1$, which by the assumption are $0, 1$ and one more exponent. By the standard way of computing the exponents (see \cite{Ince}), we see that they are solutions of 
$$
s(s-1)(s-2) - 3\tau s(s-1) + C's + G' = 0,
$$
which implies that $G'=0$ and $C'=0$. 

Now we prove that $F' = F - B\tau = 0$. A straight computation from \er{for B}, \er{for AC}, and \er{the rho8} shows that 
\begin{align*}
B &= -A - C - 2 + \rho_1 \rho_2 + \rho_1\rho_3 + \rho_2\rho_3\\
&= \g_0^u + \g_0^v + (\g_0^{u})^2 + (\g_0^{v})^2 - \g_0^u\g_0^v\\
& + \g_1^u + \g_1^v + (\g_1^u)^2 + (\g_1^v)^2 - \g_1^u\g_1^v\\
& + \ginf^u + \ginf^v - (\ginf^{u})^2 - (\ginf^{v})^2 + \ginf^u \ginf^v.
\end{align*}

Combining \er{ansie}, \er{itie}, and \er{su3int}, we see that 
\begin{equation}\label{cdgB}
\begin{split}
-\frac{1}{2}(c\g_{1, 1} + d\g_{1, 2}) &= -\frac{1}{8\pi} I_e\\
&= -((\ginf^{u})^2 + (\ginf^{v})^2 -\ginf^u \ginf^v) + (\g_0^{u})^2 + (\g_0^{v})^2 - \g_0^u \g_0^v\\
& + (\g_1^u)^2 + (\g_1^v)^2 - \g_1^u \g_0^v + \g_0^u + \g_1^u + \ginf^u + \g_0^v + \g_1^v + \ginf^v\\
& = B.
\end{split}
\end{equation}

Now we work under the assumption that $F$ is real. 
In case 1) that $\gamma_{1, 1} + \gamma_{1, 2}=-1$, which is equivalent to $\g_1^u + \g_1^v = -1$, 
we see from \er{for F} and \er{cdgB} that 
$$
F = \on{Re} F = -\frac{1}{2}\g_1^u (c\g_{1, 1} + d\g_{1, 2}) = B\tau. 
$$

In case 2) that $\g_{1, 1}=0$, similarly we have 
$$
F = \on{Re} F = -\frac{1}{2}\g_1^u(d\g_{1, 2}) = -\frac{1}{2}\g_1^u(c\g_{1, 1} + d\g_{1, 2}) = B\tau.
$$

Finally we consider case 0) that $\g_{1, 2}=0$, which means that $\g_1^u = 2\g_1^v$. 
By \er{1neq0}, we have $\tau = \g_1^u - \mu_{1, 1} = \g_1^u - (2\g_1^u - \g_1^v) - 1 = -(\g_1^v + 1)$. 
Again $C'$ and $G'$ are easily seen to be 0. From \er{for F} and \er{cdgB}, we see that 
$$
F = \on{Re} F = \frac{1}{2}(\g_1^v + 1) (c\g_{1, 1}) = \frac{1}{2}( c\g_{1, 1} + d\g_{1, 2}) (\g_1^v + 1) = B\tau
$$
for the above $\tau$. Therefore all three cases are proved to be sufficient. 
\end{proof}

Finally, we show that our reality assumption is satisfied for solutions with symmetry in accordance with that the singular sources are on the real line. 
\begin{proposition}\label{prop-real} If the solution satisfies $u(x_1, x_2) = u(x_1, -x_2)$, then the coefficients $E$ and $F$ in \er{eqsu3} are real. 
\end{proposition}

\begin{proof} 
The symmetry condition for $u$ implies that $f = e^{-u}$ is also symmetric, that is, $f(x_1, x_2) = f(x_1, -x_2)$. 
Therefore,  
$$
\p_{x_2}f\Big|_{(x_1, 0)} = 0,\quad \p_{x_2}\p_{x_1}^2 f\Big|_{(x_1, 0)} = 0,\quad \p_{x_2}^3 f\Big|_{(x_1, 0)} = 0.
$$
We consider the restriction of the ODE operator \er{eqsu3} to real independent variable, that is, $x_2 = 0$ and $z=x_1$. 
Thus in view of $\p_z = \frac{1}{2}(\p_{x_1} - i\p_{x_2})$, we see that
\begin{multline*}
\bigg(\frac{1}{8}(\p_{x_1}^3 - 3\p_{x_1}\p_{x_2}^2) + \frac{1}{2}\Big(\frac{A}{x_1^2} + \frac{B}{x_1(x_1-1)} + \frac{C}{(x_1-1)^2}\Big)\p_{x_1}\\
 + \Big(\frac{D}{x_1^3} + \frac{E}{x_1^2(x_1-1)}
+ \frac{F}{x_1(x_1-1)^2} + \frac{G}{(x_1-1)^3}\Big)\bigg) f\bigg|_{(x_1, 0)} = 0.
\end{multline*}
The function $f$ is real-valued, and by Proposition \ref{6of7}, $A$, $B$, $C$, $D$ and $G$ are real. Suppose $E=E_1 +i E_2$ and $F=F_1 + i F_2$ (actually $F_2 = -E_2$ by \er{EF}), then the imaginary part of the above equation gives
$$
\Big(\frac{E_2}{x_1^2(x_1-1)} + \frac{F_2}{x_1(x_1-1)^2} \Big) f(x_1, 0) = 0,\quad \forall x_1\neq 0, 1.
$$
Therefore $E_2 = F_2 = 0$, and $E$ and $F$ are real. 
\end{proof}

\bigskip
\end{document}